\documentclass[11pt]{amsart}

\usepackage{babel}
\usepackage{amsmath,amsthm}
\usepackage[T1]{fontenc}
\usepackage{amssymb}

\usepackage[all]{xy}

\theoremstyle{definition} 
\newtheorem{thm}{Theorem}[section]
\theoremstyle{definition}
\newtheorem{const}[thm]{Construction}
\newtheorem{cor}[thm]{Corollary}
\newtheorem{lem}[thm]{Lemma}
\newtheorem*{lem*}{Lemma}
\newtheorem*{defn*}{Definition}
\newtheorem{prop}[thm]{Proposition}
\theoremstyle{definition}
\newtheorem{defn}[thm]{Definition}
\theoremstyle{remark}
\newtheorem{rem}[thm]{Remark}
\newtheorem*{rem*}{Remark}
\newtheorem*{thm*}{Theorem}

\newtheorem{ejem}[thm]{Example}
\newtheorem{claim}{Claim}
\numberwithin{equation}{section}
\newcommand{\OO}{{\mathcal O}}
\newcommand{\A}{{\mathcal A}}
\newcommand{\U}{{\mathcal U}}
\newcommand{\PP}{{\mathcal P}}
\newcommand{\X}{{\mathcal X}}
\newcommand{\op}{{\mathrm{op}}}
\newcommand{\qc}{{\mathrm{qc}}}

\newcommand{\Y}{{\mathcal Y}}
\newcommand{\M}{{\mathcal M}}

\newcommand{\C}{{\mathfrak C}}
\newcommand{\x}{{\mathbf{x}}}

\DeclareMathOperator{\Qcoh}{{\mathbf{Qcoh}}}
\DeclareMathOperator{\Coh}{{\mathbf{Coh}}}
\DeclareMathOperator{\Spec}{\mathrm{Spec}}

\DeclareMathOperator{\colim}{\mathrm{colim}}
\DeclareMathOperator{\data}{{\text{-}\mathbf{data}}}

\newcommand{\p}{{\mathfrak p}}
\newcommand{\q}{{\mathfrak q}}

\DeclareMathOperator{\Hom}{\mathrm{Hom}}
\newcommand{\Cyl}{{\mathrm{Cyl}}}

\begin{document}
\title{Finiteness of cohomology for pro-locally proper maps}
\author{Javier Sánchez González}
\address{Departamento de Matem\'{a}ticas, Instituto de F\'isica Fundamental y Matem\'aticas, Universidad de Salamanca,
Plaza de la Merced 1-4, 37008 Salamanca, Spain}
\email{javier14sg@usal.es}
\subjclass[2010]{14A20, 14A15, 55N30, 13F30, 06A11}
\keywords{poset, schematic space, proper morphism, coherent cohomology, valuation ring}
\thanks {The author was supported by Grant PID2021-128665NB-I00 funded by MCIN/AEI/ 10.13039/501100011033}

\maketitle
\begin{abstract}
We introduce a notion of proper morphism for schematic finite spaces and prove the analogue of Grothendieck's finiteness theorem for it by means of the classic result for schemes and general descent arguments. This result also generalizes the class of morphisms of schemes for which the conclusion of the aforementioned Theorem holds. The key is giving a weaker definition of locally finitely presented morphisms.
\end{abstract}

\section{Introduction}

Schematic spaces are---finite---ringed posets that behave like qc-qs schemes with respect to their categories of quasi-coherent sheaves, they were first introduced in \cite{fernando schemes} and further studied in \cite{paper grupo etale, paper vankampen, fernando afines, fernando dualidad, notas de pedro}. The basic example comes from the construction of finite models: given a qc-qs scheme $S$ and a choice of a "locally affine" finite open cover $\{U_i\}$, one can define a finite poset $X=S/\sim$ with $s\sim s'$ iff $\cap_{s\in U_i}U_i=\cap_{s'\in U_i}U_i$ and $[s]\leq [s']$ iff $\cap_{s'\in U_i}U_i\subseteq \cap_{s\in U_i}U_i$ and a morphism of ringed spaces
\begin{align*}
\pi\colon S\to (X, \pi_*\OO_S)
\end{align*}
such that $(\pi^*\dashv \pi_*)\colon\Qcoh(S)\overset{\sim}{\to}\Qcoh(X)$ is an equivalence of abelian categories. Since $\Qcoh(S)$ determines all properties of $S$ and $\Qcoh(X)$ is easier to describe---a quasi-coherent module on a ringed poset is defined by a collection of modules at each stalk compatible with the base changes given by the partial order---, these objects are interesting to us.

In general, schematic spaces define a---non full---subcategory of ringed posets that we denote $\mathbf{SchFin}$. This subcategory is larger than the category of schemes in two senses: for any schematic space $X$ we can define a locally ringed space $\Spec(X)=\colim_{x\in X}\Spec(\OO_{X, x})$ such that, if $X$ is a finite model of some scheme $S$, there is an isomorphism $S\simeq \Spec(X)$. It happens that 1) there are schematic spaces whose spectrum is not a scheme, but still has meaningful geometry; and 2) there are schematic space whose spectrum is a scheme which cannot be obtained as finite models of it in the sense described in the previous paragraph---they may be "generalized finite models" with respect to the topology of flat monomorphisms---. We claim that most standard scheme theory can be extended to the context of schematic spaces, which is the work of the author's PhD thesis, from which \cite{paper grupo etale, paper vankampen} and this paper originate. It is worth noting that the kernel of the $\Spec(-)$ functor defines a localization of $\mathbf{SchFin}$, which makes things more subtle in practice, identifying different \textit{discretizations} of the same geometry.

\textit{Spiritually}, a morphism of schemes deserves the title of "proper" if it verifies the conclusion of Grothendieck's celebrated finiteness theorem: such a morphism $f\colon S\to T$ between locally Noetherian schemes---so that coherent modules behave well---must verify that for any coherent module $\M$ on $S$, all higher direct images $R^if_*\M$ are coherent. In categorical terms, this says that the derived direct image functor $\mathbb{R}f_*$, which does preserve quasi-coherence under quasi-compactness and quasi-separatedness hypothesis on $f$, restricts to 
\begin{align*}
\mathbb{R}f_*\colon \Coh(S)\to D_{c}(T),
\end{align*}
where $D_{c}(T)$ is the derived category of complexes with coherent cohomology.

Recall that a proper morphisms of schemes is usually defined as one verifying three properties: it must be 1) separated, 2) universally closed and 3) locally of finite presentation; where separatedness and universal closedness are usually characterized in terms of the valuative criterion---which codifies universal closedness for the specialization topology, equivalent to universal closedness for the Zariski topology under our hypothesis thanks to Chevalley's Theorem \cite[IV, Theorem 1.8.4]{EGA}---. The usual proof of the aforementioned finiteness theorem has three main ingredients: Chow's Lemma, that relates proper and projective morphisms; the \textit{d\'evissage} Lemma, that gives sufficient conditions for a subcategory of coherent sheaves to be equivalent to the whole category; and the weak version of the desired finiteness result for projective morphisms, which follows from direct calculation on finite-dimensional projective spaces.

As part of our general aim of generalizing scheme theory to the schematic world, are interested in giving a definition of proper morphism for schematic spaces that---at least in many cases---\textit{discretizes} the scheme-theoretic notion and for which the analogous finiteness result holds. The result we arrive to not only achieves this, but it also provides a generalization of the standard definition for schemes by admitting a weaker version of the "finite type" condition adapted to the topology of flat monomorphisms. We will call these new morphisms \textit{pro-locally proper}. At this stage, the proof of the finiteness theorem can be carried out in two ways. The first method would be adapting Grothendieck's original strategy, which is feasible because we have a theory of finite projective spaces \cite{fernando universal} for which finiteness of cohomology holds and the \textit{d\'evissage} Lemma is a categorical statements whose proof works in our situation with little change. We take, however, a second route: if we assume that Grothendieck's result holds for schemes, then we can use a descent argument in the topology of flat immersions of schematic spaces. We will make use of two techniques: 1) descent arguments in terms of the cylinder spaces introduced in \cite{paper vankampen}, to pass from a "discrete" characterization of pro-locally proper map to one in scheme-theoretic terms; and 2) an elementary study of the theory of \textit{proschemes} and \textit{algebraic proschemes}, which will be the geometric objects \textit{represented by} schematic spaces (see Section \ref{section cohomology}). Actually, we will provide two different definitions of "properness" for schematic spaces, which do not necessarily coincide in full generality. One of them extends the definition for schemes in the separated case, while the other one does not.

The paper is structured as follows: we first give a brief introduction to the theory of schematic spaces required to understand the context of our discussion, which includes recalling the construcion of the cylinder space and some facts about the topology \textit{of flat immersions} that is natural in the world of schematic spaces---some definitions will be presented in a novel incarnation adapted to the discussion of this paper---; then we see how the valuative criterion-type of definitions behave well on schematic spaces while, unfortunately, defining morphisms of finite presentation or finite type presents various issues. Next, we define \textit{pro-locally of finite presentation} and \textit{pro-locally proper} morphisms as a solution to these issues and characterize them in more explicit terms via the cylinder space construction. After that, we digress into the more technical part of the paper in Sections \ref{section algebraic} and \ref{section cohomology}: we introduce a notion of \textit{algebraic morphism} inspired by that of \textit{representable functor} and use it to give an alternative definition for properness: \textit{algebraically proper morphisms}; but more importantly, this notion of algebraicity is used to relate schematic spaces that are not finite models of schemes and those that are, in form of Theorem \ref{theorem b422} and its consequences related to preservation of cohomology. Finally, we prove the promised finiteness Theorem for both pro-locally proper morphisms and algebraically proper morphisms.
\section{Schematic spaces, flat immersions and the cylinder space}

The basic theory of schematic spaces is dealt with from a cohomological point of view in \cite{paper grupo etale, fernando schemes, fernando dualidad} and from a combinatorial point of view in \cite{paper vankampen} or \cite{notas de pedro}. For our purposes here, we will prove that these are equivalent to a third equivalent formulation, which has a more "synthetic" flavour.

First, let $\mathbf{CRing}\data$ be the category of ringed finite posets---notation explained in \cite{paper vankampen}--- and $\mathbf{LRS}$ the category of locally ringed spaces. We extend the ordinary prime spectrum functor as follows:
\begin{align*}
\Spec\colon\mathbf{CRing}\data&\to \mathbf{LRS} & & X\mapsto\colim_{x\in X}\Spec(\OO_{X, x}).
\end{align*}

Note that if $X$ is a finite model of a qc-qs scheme $S$---as outlined in the Introduction---this construction recovers the original scheme. Furthermore, if $X$ has restriction morphisms, $r_{xy}\colon\OO_{X, x}\to\OO_{X, y}$ for any $x\leq y$, that are flat ring epimorphisms of finite presentation---i.e. they induce open immersions between their prime spectra---, $\Spec(X)$ is also a scheme. We will, however, admit a more general situation that will allow us to "model" schemes, but with respect to generalized coverings; as well as other locally ringed spaces that still present an "scheme-like" behaviour.

\begin{defn}
Let $X$ be a ringed poset. We say that $X$ is pseudoschematic if its restriction morphisms are flat ring epimorphisms. 
\end{defn}
\begin{rem}
Recall that a faithfully flat ring epimorphism is an isomorphism. This fact---essentially---implies that, if we have a commutative triangle of locally ringed spaces
\begin{align*}
\xymatrixrowsep{0.05in}\xymatrix{ 
\Spec(A)\ar[rd]\ar[dd] & \\
& S \\
\Spec(B)\ar[ru] &
}
\end{align*}
where all morphisms are flat monomorphisms and where both arrows into $S$ define the same \textit{topological image}, then $A\simeq B$. In other words, \textit{any subset of $S$ that is the image of a flat monomorphism}---for instance, an arbitrary intersection of Zariski opens---comes from a uniquely determined---up to isomorphism---flat monomorphism, defined by restricting the sheaf of rings to the subset. Morally, this tells us that the topology of flat monomorphisms is \textit{the weakest topology in which it is possible to fully study locally affine spaces from their locally ringed space incarnation}. This explains why, for example, the locally ringed space associated to an algebraic space is not that relevant for its geometry: since étale morphisms are not monomorphisms, the corresponding sheaf of rings does not carry enough algebraic information. For practical use of this property, see Section \ref{section cohomology}.
\end{rem}

\begin{defn}\label{defn schematic}
Let $X$ be a pseudoschematic space. We say that $X$ is \textit{schematic} if for all $x, y, z\in X$ with $x, y\geq z$, the natural morphism
\begin{align*}
\OO_{X, x}\otimes_{\OO_{X, z}}\OO_{X, y}\to\prod_{t\geq x, y}\OO_{X, y}
\end{align*}
is faithfully flat.
\end{defn}

The next Theorem gives the promised "synthetic characterization": a pseudoschematic space is schematic if it is---in a sense that can be made more precise---a finite model of its spectrum.
\begin{thm}\label{theorem characterization schematic}
Let $X$ be a pseudoschematic space. Then $X$ is {schematic} if and only if every point $\x=[(x, \p_x)]\in\Spec(X)$ has a unique representative whose topological component $\pi_X(x)$ is maximal with respect to the partial order of $X$. In short, if there exists a \textit{centre map}
\begin{align*}
\pi_X\colon \Spec(X)\to X.
\end{align*}
\begin{proof}
Note that, when $X$ is pseudoschematic, the restriction morphisms induce flat monomorphisms of affine schemes between the prime spectra and, thus, the underlying topological spaces of fibered products between those spectra coincide with the topological fibered products.


Now, let us assume that $\pi_X$ exists and prove that for any $z\leq x, y$ the map $\coprod_{t\geq x, y}\Spec(\OO_{X, t})\to \Spec(\OO_{X, x})\times_{\Spec(\OO_{X, z})}\Spec(\OO_{X, y})$ is surjective. Indeed, a point $(\p_x, \q_y)\in \Spec(\OO_{X, x})\times_{\Spec(\OO_{X, z})}\Spec(\OO_{X, y})$ defines a point in $\Spec(X)$ which, by the hypothesis, gives us some unique $t\in X$ and $\alpha_t\in\Spec(\OO_{X, y})$ such that $t\geq x, y$ and $\alpha_t\mapsto (\p_x, \q_y)$ via the map above. For the converse, follow the argument backwards. 
\end{proof}
\end{thm}

Originally, schematic morphisms were also defined via cohomological means. When both source and target space are schematic, their description can be simplified and even given in terms of properties of ring morphisms---see \cite[Definition 2.1]{paper vankampen}---. By an argument very similar to the one of Theorem \ref{theorem characterization schematic}, this last description can be seen to be equivalent to the following:
\begin{defn}\label{definition schematic morphism}
Let $X, Y$ be schematic spaces. A morphism $f\colon X\to Y$ is \textit{schematic} if and only if the diagram
\begin{align*}
\xymatrix{ 
\Spec(X)\ar[d]_{\pi_X}\ar[r]^{\Spec(f)}& \Spec(Y)\ar[d]^{\pi_Y}\\
X\ar[r]^f& Y
}
\end{align*}
commutes (in $\mathbf{Set}$).
\end{defn}
\begin{rem}
Note that, for any morphism of ringed spaces between schematic spaces, the diagram of Definition \ref{definition schematic morphism} verifies $f\circ \pi_X\leq \pi_Y\circ \Spec(f)$. In other words, \textit{$f$ is schematic if and only if it preserves "centres"}; so we may say that schematic morphisms are \textit{central}.
\end{rem}

An important subclass of schematic spaces is that of \textit{affine} schematic spaces. Different perspectives on this notion can be found in \cite{paper vankampen, fernando afines, fernando schemes}:
\begin{defn}
A schematic space $X$ is said to be \textit{affine} if its global sections functor $\Qcoh(X)\to \Qcoh(\OO_X(X))$ is an equivalence of categories. A morphism $f\colon X\to Y$ is said to be \textit{affine} if $f^{-1}(U_y)$ is affine for all $y\in Y$.
\end{defn}

With help of Corollary \ref{corollary useless}, presented later in this paper, one proves:
\begin{prop}\label{prop affine characterization}
A schematic space $X$ is affine iff $\Spec(X)\simeq \Spec(\OO_X(X))$. 
\begin{proof}
The \textit{only if} part is a particular case of  \cite[Proposition 6.6]{fernando schemes}---a beautiful construction that generalizes to other situations, partially discussed in \cite{paper vankampen}---. For the converse, note that $\Spec(X)\simeq \Spec(\OO_X(X))$  gives an equivalence $\mathbf{dQcoh}(\Spec(X))\simeq \mathbf{dQcoh}(\Spec(\OO_X(X))$, but since both are schemes, these categories are isomorphic to their ordinary categories of quasi-coherent modules by Corollary \ref{corollary useless}.
\end{proof}
\end{prop}
Now, let $\mathbf{SchFin}$ be the category of schematic spaces and morphisms. The $\Spec$ functor restricts to $\Spec\colon\mathbf{SchFin}\to \mathbf{LRS}$. One can check that $\mathbf{SchFin}$ has finite fibered products that commute with $\Spec$ \textit{and} with the forgetful to ringed spaces, so it relates products of ringed spaces with those of locally ringed spaces, which does have interest on its own, see \cite{Gillam}. We sketch the proof of the first claim because of its importance to make sense of Section \ref{section algebraic}:
\begin{prop}\label{prop commutation}
The $\Spec$ functor commutes with finite fibered products. 
\begin{proof}
We want to see that $\Psi\colon \Spec(X\times_ZY)\to \Spec(X)\times_{\Spec(Z)}\Spec(Y)$ is an isomorphism for any $X, Y, Z$. If all three spaces are affine, this is trivial, because the global sections ring of $X\times_ZY$ is the tensor product of the three global sections rings (C.f. \cite{fernando schemes}). In general, one first proves that it is bijective, using that points of the target space are 4-tuples $(\x, \mathbf{z}, \mathbf{y}, \alpha)$ with $(\x, \mathbf{z}, \mathbf{y})$ belonging to the topological fibered product and $\alpha$ a prime ideal of $\kappa(\x)\otimes_{\kappa(\mathbf{z})}\kappa(\mathbf{y})$, where $\kappa(\x)=\OO_{X, x}/\p_x$ is the residue field of the point $\x$ with maximal representative $(x, \p_x)$---as in Theorem \ref{theorem characterization schematic}, same for $\mathbf{y}$ and $\mathbf{z}$---; and that points of the source space are equivalence classes of points $(x, z, y, \beta)$ with $(x, y, z)$ a point of $|X\times_ZY|$ and $\beta$ a prime of $\OO_{X, x}\otimes_{\OO_{Z, z}}\OO_{Y, y}$, which is in itself given by a prime of each ring and one of the product of residue fields. The bijection follows from maximality arguments as in Theorem \ref{theorem characterization schematic}. Finally, to see that it enhances to an isomorphism of locally ringed spaces, one argues that, by the affine case and the fact that $\Psi$ is bijective, $U_{(x, y, z)}\to  U_x\times_{U_z}U_y$ induce isomorphisms on spectra for all $(x, z, y)$ and cover the entire space, hence they provide isomorphisms at stalks that prove the claim.
\end{proof}
\end{prop}

Since there are non-isomorphic schematic spaces with isomorphic spectra, we shall localize our category by the "kernel" of $\Spec$.
\begin{defn}
A schematic morphism $f$ is said to be a \textit{qc-isomorphism} if $\Spec(f)$ is an isomorphism.
\end{defn}
This class of morphisms clearly defines a---maximal---left multiplicative system of arrows. Note that $X$ is affine if and only if $X\to (\star, \OO_X(X))$ is a qc-isomorphism. In general, working as in the proof of Proposition \ref{prop affine characterization}, one characterizes these as maps inducing equivalences of quasi-coherent modules \textit{or} affine morphisms that are isomorphisms at the level of sheaves of rings.  Denote the corresponding Verdier quotient by $\mathbf{SchFin}_{\qc}$ and consider
\begin{align*}
\Spec\colon\mathbf{SchFin}_{\qc}\to\mathbf{LRS}
\end{align*}
which is now faithful by construction, but \textit{a priori} not full---we will obtain fullness under additional "algebraicity" restrictions discussed in Sections \ref{section algebraic} and \ref{section cohomology}---. The paradigm for what constitutes a "good" definition or property in the schematic category is, according to our judgement, the following: 
\begin{itemize}
\item Our "geometric" objects of study are the locally ringed spaces associated to schematic spaces via $\Spec$, thus in order to define properties of schematic spaces or morphisms, we---generally---want them to factor through $\mathbf{SchFin}_{\qc}$. We say that such properties are \textbf{geometric}.

\item An easy way out of the previous point would be declaring that an object verifies a property if and only if its spectrum does. The problem is that this process defeats the purpose of working with finite models as a simplification of less explicit spaces. As such, we want our properties to \textbf{discrete} \textit{as often as possible}, i.e. to be able to judge whether or not they hold without applying the $\Spec$ functor. We just remark that, in some cases not discussed here, this process requires considering certain reflexive subcategories of the schematic category, see for example the discussion on connectedness on \cite{paper grupo etale}.
\end{itemize}

\subsection{Flat immersions and cylinder spaces}

The following results are extracted from the prequel \cite{paper vankampen} and we refer there for a fuller exposition in terms of the notion of \textit{category of $\C\data$}. We shall avoid such terminology here and simply state the most elementary facts we need.
\begin{defn}
A morphism $f\colon X\to Y$ is a \textit{flat immersion} if it is flat, i.e. $f\colon\OO_{Y, f(x)}\to\OO_{X, x}$ for all $x\in X$; and its relative diagonal $\Delta_f\colon X\to X\times_YX$ is a qc-isomorphism. A finite family of flat immersions $\{f_i\colon X_i\to Y\}$ is a cover of $Y$ if $g\colon \coprod_iX_i\to Y$ is faithfully flat, i.e. if, additionally, the map $\OO_{Y, y}\to\prod_{f_i(x_i)\geq y}\OO_{X_i, x_i}$ induces a surjection of prime spectra. 
\end{defn}

Now, let $P$ be a finite poset and $\X\colon P\to \mathbf{SchFin}^\op$ be a functor. In \cite{paper vankampen} this is called a $\mathbf{SchFin}^\op$-datum. The following construction is a explicit presentation of the \textit{lax colimit} of $\X$, seen as a pseudofunctor with respect to the natural 2-categorical structure on $\mathbf{SchFin}$. For each $p\leq q$ in $P$ let $\X_{pq}\colon\X(q)\to \X(p)$ denote the induced morphism. A point of $\X(p)$ will be denoted by $x_p$ to indicate the space it originally belongs to.

\begin{defn}
Given $\X$ as above, we define the \textit{cylinder of $\X$} as the ringed poset $\Cyl(\X)$ such that:
\begin{itemize}
\item Its underlying set is $\coprod_{p\in P}|\X(p)|$. It is considered as a poset with the partial order defined by:
\begin{align*}
x_p\leq y_q \Leftrightarrow p\leq q \text{ in }P\text{ and }p\leq \X_{pq}(q)\text{ in }\X(p)
\end{align*}
\item Its sheaf of rings has stalk rings $\OO_{\Cyl(\X), x_p}=\OO_{\X(p), x_p}$ and restriction maps induced by those of each $\X(p)$ and the natural morphisms of sheaves of rings $\X^\sharp_{pq}$ for all $p\leq q$.
\end{itemize}
\end{defn}

Now, in \cite{paper vankampen} we give a characterization of functors $\X$ whose cylinder space is schematic. The relevant corollary is the following:
\begin{prop}\cite[Corollary 6.14]{paper vankampen}\label{prop nerve}
Let $X$ be schematic, $\{U_i\to X\}_{i\in I}$ a finite \textit{family} of flat immersions and let $\U\colon \PP^*(I)\to \mathbf{SchFin}^\op$ be the \textit{nerve} functor---with source the non-empty parts of $I$---, which is defined by $\U(\Delta)=\prod_{i\in \Delta}U_i$---fibered product over $X$---. Then $\Cyl(\U)$ is schematic and comes equipped with a natural flat immersion $\Cyl(\U)\to X$ that is a qc-isomorphism if and only if the original family was a covering.
\end{prop}

\begin{rem}
The cylinder construction is functorial in the following sense: a morphism of "data" $f\colon\Y\to \X$ with respective source posets $|\Y|$ and $|\X|$ is given by a continuous map (a functor) $|f|\colon |\Y|\to|\X|$ between the underlying posets and a collection of compatible morphisms $f_p\colon \Y(p)\to \X(f(p))$ for all $p\in|\Y|$. Such a morphism induces $\Cyl(f)\colon\Cyl(\Y)\to \Cyl(\X)$, as it can be readily checked. Furthermore, \textit{if both source and target are schematic, and $|f|$ is the identity map, $\Cyl(f)$ will be schematic.}
\end{rem}

\section{Valuative criteria and v-proper morphisms}
\textit{For the sake of reducing ourselves to the case of discrete valuation rings, we will assume that all spaces in this section are locally Noetherian.} A general theory may be developed by admitting infinite schematic spaces.
\smallskip

Let $A$ be a discrete valuation ring with fraction field $\Sigma$. The prime spectrum $\Spec(A)$ is a scheme, but also a schematic space with underlying poset $\{0\leq 1\}$ and sheaf of rings $\OO_{\Spec(A), 0}=A$ and $\OO_{\Spec(A), 1}=\Sigma$. The locally ringed space associated to it via the $\Spec\colon\mathbf{SchFin}\to\mathbf{LRS}$ functor is $\Spec(A)$ itself, i.e. abusing notation we can write 
\begin{align}\label{equation}
\Spec(\Spec(A))=\Spec(A).
\end{align}

The previous discussion also works for a field $A$---which gives a nice theory of geometric points---and, in general, the arguments that follow work for any scheme $S$ that has the specialization topology and verifies equation \ref{equation}.

\begin{lem}\label{lemma specialization topology adjunction}
Let $\mathbf{CRing}\data$ denote the category of ringed posets and $\mathbf{RS}_0$ the category of T$_0$ ringed spaces. The inclusion functor
\begin{align*}
\mathbf{CRing}\data\hookrightarrow \mathbf{RS}_0
\end{align*}
admits a right adjoint $(-)_\leq$, mapping each $X\in\mathbf{RS}_0$ to the poset $X_\leq$ given by the specialization partial order with sheaf of rings given by the stalk rings.
\begin{proof}
This is just an enhancement to the ringed setting of the well known property of the specialization partial order. Details to the reader.
\end{proof}
\end{lem}
\begin{lem}\label{lema qc iso to valuation}
Let $A$ be a field or discrete valuation ring. Any qc-isomorphism $f\colon Z\to \Spec(A)$ is an isomorphism. 
\begin{proof}
The inverse is given by $\pi_Z\circ (\Spec(f)^{-1})_{\leq}$, with $\pi_Z$ the centre map and $(-)_\leq$ the adjoint of Lemma \ref{lemma specialization topology adjunction}. This morphism verifies Definition \ref{definition schematic morphism} by construction, so it is schematic.  Note that $\pi_Z$ is continuous in the specialization topology.
\end{proof}
\end{lem}

\begin{prop}\label{prop localization valuations}
For any field or discrete valuation ring $A$ and schematic space $X$ such that $\Spec(X)$ is a scheme, the localization  functor $\mathbf{SchFin}\to\mathbf{SchFin}_\qc$ induces isomorphisms
\begin{align*}
\Hom_{\mathbf{SchFin}}(\Spec(A), X)\overset{\sim}{\to}\Hom_{\mathbf{SchFin}_{\qc}}(\Spec(A), X).
\end{align*}
\begin{proof}
It follows from Lemma \ref{lema qc iso to valuation}.
\end{proof}
\end{prop}
\begin{prop}\label{cor valuation schemes and schematic}
For any field or discrete valuation ring $A$ and schematic space $X$, the $\Spec$ functor induces isomorphisms
\begin{align*}
\Hom_{\mathbf{SchFin}}(\Spec(A), X)\overset{\sim}{\to}\Hom_{\mathbf{Sch}}(\Spec(A), \Spec(X)).
\end{align*}
\begin{proof}
Taking finite models on any $\Spec(A)\to \Spec(X)$ produces a morphism in $\mathbf{SchFin}_{\qc}$, and we conclude again by Lemma \ref{lema qc iso to valuation}.
\end{proof}
\end{prop}
\begin{rem}
For finite models of schemes, the statement of Proposition \ref{cor valuation schemes and schematic} is a corollary of Proposition \ref{prop localization valuations} and the fact that the finite model functor $\mathbf{Sch}\to\mathbf{SchFin}_{\qc}$ is fully faithful, but this argument does not work when $X$ is not a finite model of any $S$, even if $\Spec(X)\simeq S$ is a scheme. It will work for "algebraic schematic spaces" by Theorem \ref{theorem b422}.
\end{rem}

This motivates the following Definition:
\begin{defn}\label{definition vproper}
We say that a schematic morphism $f\colon X\to Y$ is \textit{v-separated} (resp. \textit{v-universally closed, v-proper}) if for every discrete valuation ring $A$ with fraction field $\Sigma$ the natural morphism
\begin{align*}
\Hom_{\mathbf{SchFin}_{/Y}}(\Spec(\Sigma), X)\to \Hom_{\mathbf{SchFin}_{/Y}}(\Spec(A), X)
\end{align*}
is inyective (resp. surjective, bijective).
\end{defn}

Note that Definition \ref{definition vproper} is \textit{discrete}, in the sense that it does not require us to talk about the $\Spec$ functor or even localization. Furthermore, thanks to Propositions \ref{prop localization valuations} and \ref{cor valuation schemes and schematic}, it follows that:  
\begin{lem}
\textit{Being v-separated (resp. v-universally closed, v-proper)} is a geometric property of morphisms. In particular, if $f\colon X\to Y$ is a morphism such that $\Spec(f)$ is a morphism of schemes, then $f$ is v-separated (resp. v-universally closed, v-proper) if and only if $\Spec(f)$ verifies the uniqueness part (resp. uniqueness part, both parts) of the valuative criterion.
\end{lem}	

Now we should prove that our valuative criterion-based definitions coincide with more ordinary ones. A complete proof requires a lot of routine work that exceeds our purposes here and will be fully presented in future papers, so we limit ourselves to highlight the main ideas, which are parallel to those of scheme theory.
\begin{defn}
A schematic morphism $f\colon X\to Y$ is a \textit{closed immersion} if it is affine and $f_\sharp\colon\OO_Y\to f_*\OO_X$ is an epimorphism of modules (i.e. surjective). A schematic morphism if said to be \textit{separated} if its relative diagonal $\Delta_f\colon X\to X\times_YX$ is a closed immersion. A schematic space $X$ is said to be \textit{separated} if the projection $X\to \star$ is separated.
\end{defn}
\begin{rem}\label{remark separated morphisms}
If $f\colon X\to Y$ is separated with $Y$ separated and $U, V\to X$ are flat immersions with $U, V$ affine spaces, then $U\times_XV$ is affine and its ring of global sections is a quotient of $\OO(U)\otimes_{\OO_{Y}(Y)}\OO(V)$. Indeed, $U\times_YV$ is affine with global sections being a quotient of $\OO(U)\otimes_{\OO_{Y}(Y)}\OO(V)$ by the hypothesis on $Y$; and finally $U\times_XV\simeq X\times_{X\times_YX}(U\times_YV)$, which is also affine with its ring of global sections being a further quotient of $\OO(U)\otimes_{\OO_{Y}(Y)}\OO(V)$ by the hypothesis on $f$.
\end{rem}
\begin{thm}
A schematic morphism $f\colon X\to Y$ is separated if and only if it is v-separated.
\begin{proof}[Ideas of the proof]
In general, the diagonal map $\Delta_f$ is always an \textit{immersion}, which shall be understood \textit{with respect to the topology of flat immersions}, i.e. there exists a \textit{flat} immersion $W\to X\times_YX$ such that the base change of $\Delta_f$ is a closed immersion and such that $W\times_{X\times_YX}X\to X$ is a qc-isomorphism. This would be the first Lemma for the proof. Now, note that, under the pseudoschematic hypothesis, one can prove that closed subsets of $\Spec(X)$ are in bijective correspondence with sheaves of quasi-coherent radical ideals of $\OO_X$, as it would happen for schemes, so we are allowed to define \textit{morphisms with closed image} in terms of these sheaves of ideals of $X$, thus the second Lemma would be: \textit{an immersion with closed image is a closed immersion}. This is non-trivial.

Finally, it is easy to see that the v-separatedness condition is equivalent to v-universal closedness of the diagonal, thus if we see that v-universally closed morphisms---closed for the specialization topology---have closed image---and more general, \textit{are} universally closed---, we are done by the Lemmas of the previous paragraph. To prove this, one notes that it holds for morphisms between affine schematic spaces---as it does for affine schemes---, so it does hold locally; and then extends the result via direct images through the natural inclusions, which preserve quasi-coherence, and computations at stalks. It is analogue to the proof for schemes in sheaf-theoretic language.
\end{proof}
\end{thm}

If we restrict ourselves to separated spaces---and morphisms---, we can define finite models of schematic spaces with respect to flat monomorphisms, obtaining qc-equivalent objects and maps that have specified stalk rings.

\begin{const}[Finite model with respect to flat immersions]\label{const finite model flat imm}
Given a separated schematic space---or even less, a \textit{semiseparated} space: one with affine diagonal map---, a cover by flat immersions $\{X_i\to X\}$ defines a nerve $\X$ and a qc-isomorphism $\Cyl(\X)\to X$, but, on the other hand, we also have a datum $\OO(\X)$ with $(\star, \OO_{\X(\Delta)}(\X(\Delta)))$ at each $\Delta$. The morphism of data $\X\to \OO(\X)$ is a qc-isomorphism at each $\Delta$ because all $\X(\Delta)$ are affine by the separatedness hypothesis. One can check that $\Cyl(\OO(\X))$ is schematic and that $\Cyl(\X)\to \Cyl(\OO(\X))$ is a qc-isomorphism---this follows easily from $\X\to \OO(\X)$ being the identity map between posets and a qc-isomorphism at each $\Delta$---, thus we shall define $\Cyl(\OO(\X))$ to be the "finite model of $S$ with respect to the given cover". Note that, while this construction provides a schematic space with the desired stalk rings, it does not exactly mirror the ordinary construction, since it does not take the topology of the original space into account: the underlying poset of $\Cyl(\OO(\X))$ is the non-empty parts of the set indexing the covering. This construction generalizes to separated morphisms $f\colon X\to Y$ with $Y$ separated in a straightforward way. 
\end{const}

\section{Finite presentation and proper maps}

If one attempts to define schematic morphisms of finite presentation at the level of stalk rings---as we did for flat morphisms---, problems arise. For example, let us temporarily say that $f\colon X\to Y$ is of finite presentation when $f^\sharp_x\colon\OO_{Y, f(x)}\to \OO_{X, x}$ are of finite presentation for all $x\in X$. If we compose with a qc-isomorphism $\phi\colon Y\to Z$, since the morphisms $\phi_y^\sharp$ are only flat ring epimorphisms, we cannot guarantee that $\phi\circ f$ remains of "finite presentation", i.e. such naive definition is not geometric. For similar reasons, this notion is not stable under base change either.

The root of the issue here is that, while the property of ring homomorphisms of \textit{being of finite presentation} is local in the Zariski topology---\cite{stacks}---, it is \textit{not} local in the topology of flat monomorphisms. For properties "local with respect to flat monomorphisms"---such as \textit{being flat} or being \textit{weakly étale}---it is possible to define things stalk-wise and obtain a number of equivalent characterizations analogous to the usual ones in scheme theory. For properties that do not behave as well, things get a little more subtle. Fortunately, in good cases, our \textit{a priori} indirect definition will specialize to a familiar statement of "local nature" in the topology of flat immersions.
\smallskip

For convenience, let us introduce the following terminology.
\begin{defn}
Let $\C$ any category. We say that two arrows $f\colon c\to d$ and $f'\colon c'\to d'$ are \textit{equivalent}, written $f\simeq g$, if they coincide in any skeleton $\mathrm{sk}(\C)$ of $\C$. In other words, if there are isomorphisms $h\colon c'\to c$ and $g\colon d'\to d$ in $\C$ such that $f\circ h=g\circ f'$. 
\end{defn}

\begin{defn}
We say that two schematic morphisms $f\colon X\to Y$ and $g\colon Z\to T$ are \textit{qc-equivalent} if they induce equivalent arrows in $\mathbf{SchFin}_{\qc}$. We denote this equivalence relation by "$f\overset{\qc}{=}g$".
\end{defn}

In particular, two morphisms $f, g\colon X\to Y$ with the same source and target are qc-equivalent iff there are qc-isomorphisms $\phi, \varphi\colon Z\to X$ with $f\circ \phi=g\circ \varphi$. Also note that a schematic morphism is a qc-isomorphism if and only if it is qc-equivalent to the identity. 

\begin{defn}\label{defn geometric property}
We say that a property $\mathbf{P}$ of morphisms of schematic spaces is \textit{geometric} if for any morphisms $f$ and $g$ with $f\overset{\qc}{=}g$, then $\mathbf{P}(f)\Leftrightarrow\mathbf{P}(g)$.
\end{defn}

Definition \ref{defn geometric property} can be characterized as a series of stability and closure conditions under base change and composition with qc-isomorphisms. As such fact is will not be relevant here, we leave it without further mention. In any case, we can use this notion as inspiration for our general definitions:

\begin{defn}\label{defn relative scheme}
Let $Y$ be a schematic space. We say that a morphism $f\colon X\to Y$ is a \textit{scheme fibration}---or that $X$ is a scheme over $Y$---if it is qc-equivalent to some $f'\colon X'\to Y'$ such that $\Spec(f'^{-1}(U_{y'}))$ is a scheme for every $y'\in Y'$.
\end{defn}
\begin{ejem}
Any affine schematic morphism is a scheme fibration.
\end{ejem}

\begin{defn}\label{defn prolocally P}
Let $\mathbf{P}$ be a property of morphisms of schemes. We say that a scheme fibration $f\colon X\to Y$ as in Definition \ref{defn relative scheme} is \textit{pro-locally $\mathbf{P}$} if, for all $y'\in Y'$, $\mathbf{P}(\Spec(f^{-1}(U_{y'})\to \Spec(U_{y'}))$.
\end{defn}

By construction, we have
\begin{lem}
For any property $\mathbf{P}$, \textit{being pro-ĺocally $\mathbf{P}$} is geometric.
\end{lem}
\begin{lem}
A morphism $f\colon X\to Y$ is pro-locally proper if and only if it is pro-locally of finite presentation and v-proper.
\end{lem}

In particular, Definition \ref{defn prolocally P} produces a notion of morphisms \textit{pro-locally of finite presentation}---or finite type---that is geometric. Now we see that under some additional hypothesis, such statement can be translated to a local one in a vein analogue to scheme theory, purely in terms of ring maps and schematic morphisms, i.e. to a discrete statement.

\begin{thm}\label{theorem prolocal}
A separated morphism $f\colon X\to Y$ with $Y$ a separated space is pro-locally of finite presentation if and only if it admits finite coverings by flat immersions $\{U_i\to X\}_{i\in I}$ and $\{V_j^i\to U_i\times_YX\}_{j\in J_i}$ such that the ring maps $\OO(U_i)\to \OO(V_j^i)$ are of finite presentation for all $i, j$.
\begin{proof}
If $f$ is pro-locally of finite presentation, since it is a scheme fibration, there is a diagram
\begin{align*}
\xymatrix{ 
X'\ar[d]_{f'} & X''\ar[l]\ar[r]\ar[d]_{f''} & X\ar[d]^f\\
Y' & Y''\ar[l]\ar[r] & Y
}
\end{align*}
where the horizontal arrows are qc-isomorphisms and $f'$ is as in Definition \ref{defn prolocally P}. Note that we are allowed to assume that $f''$ exists---which is not a consequence of the mere definition of qc-equivalence---because, if it did not, we could simply replace $X''$ by $X''\times_X(Y''\times_YX)$ and the suitable morphisms. Now, we can consider the covering $\{U_{y'}\to Y'\}_{y\in Y'}$ of $Y'$, which induces a covering $\{U_{y'}\times_{Y'}Y''\to Y''\to Y\}_{y'\in Y'}$ of $Y$. As for the second covering, since $\Spec(f^{-1}(U_{y'}))$ are schemes locally of finite presentation over the base for all $y'$, we can take finite models $X'_{y'}$ of each $f'^{-1}(U_{y'})$ and obtain affine covers $\{U_{x'}\subseteq X_{y'}'\to X'\}_{x'\in X'_{y'}}$ such that $\OO_{Y', y'}\to \OO_{X'_{y'}, x'}$ are of finite presentation. Base changing these to $X''$ and composing with the qc-isomorphism $X''\to X$ as we did before, we win.

For the converse, let us see that, if coverings as in the statement exist, there is a finite affine covering by flat immersions $\{W_k\to Y\}_k$ such that $\Spec(W_k\times_YX)$ are \textit{schemes} locally of finite presentation over $\Spec(W_k)$ for all $k$. Then, Construction \ref{const finite model flat imm} would complete the proof. 

We proceed with the definition of $\{W_k\to Y\}_{k\in K}$. Set $K=\PP^*(I)$ and $W_k=\U(\Delta)\to Y$---the images of the nerve $\U$ as in Proposition \ref{prop nerve}---. By the hypothesis, $\U(\Delta)\times_YX$ admits a cover $\{V^\Delta_j\to \U(\Delta)\times_YX\}_{j\in J_\Delta}$, where $V^\Delta_j=\prod_{i\in \Delta}V^i_j$---fibered product over $X$---and $J_\Delta=\prod_iJ_i$. We claim that, since both $Y$ and $f$ are separated, $\OO(\U(\Delta))\to \OO(V_j^\Delta)$ remain of finite presentation for each $\Delta$ and all $j$. 
Indeed, the separatedness conditions imply that any tensor product $\OO(V_j^i\times_X V_{j'}^{i'})$ is a quotient of $\OO(V_j^i)\otimes_{\OO(Y)}\OO(V_{j'}^{i'})$ and, similarly, $\OO(U_i\times_YU_{i'})$ is a quotient of $\OO(U_i)\otimes_{\OO(Y)}\OO(U_{i'})$. One obtains commutative diagrams of rings
\begin{align*}
\xymatrix{ 
\OO(V_j^i\times_X V_{j'}) & \OO(V_j^i)\otimes_{\OO(Y)}\OO(V_{j'}^{i'})\ar[l]\\
\OO(U_i\times_YU_{i'})\ar[u]  & \OO(U_i)\otimes_{\OO(Y)}\OO(U_{i'})\ar[u] \ar[l]
}
\end{align*}
where the horizontal arrows are quotients---hence \textit{surjective}---and the right vertical arrow is of finite presentation due to being a product of two such maps. An easy commutative algebra exercise proves that the remaining arrow is of finite presentation (no Noetherian hypothesis required).

Note that, by construction, the restriction morphisms between the $\OO(V_j^\Delta)$ as $\Delta$ varies are also flat epimorphisms of finite presentation, thus they induce open immersions of affine schemes. In particular, since using Construction \ref{const finite model flat imm} we can define a finite model of each $\U(\Delta)\times_YX$ whose stalk rings are the $\OO(V_j^\Delta)$, the standard \textit{recollement} argument for schemes and open immersions proves that $\Spec(\U(\Delta)\times_YX)$ is a scheme for all $\Delta$, and one of finite presentation over $\Spec(\U(\Delta))$ by construction, as desired.
\end{proof}
\end{thm}

In other words, if we consider the morphism of locally ringed spaces $\Spec(f)\colon\Spec(X)\to \Spec(Y)$ represented by $f$, Definition \ref{defn prolocally P} for the finite presentation property states that there exists \textit{a certain} covering by flat monomorphisms of $\Spec(Y)$ with respect to which $\Spec(f)$ is "locally of finite presentation" in the ordinary sense of the expression. However, we cannot arbitrarily replace this cover by another one, because the transition morphisms involved in such process are flat monomorphisms and thus not "compatible" with properties that are not pro-local.

Consequently, it makes sense to give the following definition for schemes. 
\begin{defn}\label{prolocal schemes}
A morphism of schemes $g\colon S\to T$ is pro-locally proper if there exists a pro-locally proper schematic morphism $f\colon X\to Y$ such that $\Spec(f)\simeq g$---i.e. there are isomorphisms of schemes making the suitable diagram commutative---.
\end{defn}
The following result follows from Theorem \ref{theorem prolocal}. 
\begin{prop}\label{prop proplocally proper schemes}
If $g\colon S\to T$ is pro-locally of finite presentation (resp. proper), it admits a finite affine covering $\{U_i\to T\}$ by flat monomorphisms such that $U_i\times_TS\to U_i$ are morphisms locally of finite presentation (resp. proper). The converse holds if $g$ and $T$ are separated.
\end{prop}
\begin{rem}
\textit{A priori}, the converse of Proposition \ref{prop proplocally proper schemes} may not hold in full generality, because we do not have a way of constructing finite models with respect to coverings by flat monomorphisms. This is tied with deep questions about the theory of schematic spaces and the associated spectra, which we call proschemes in Section \ref{section cohomology}. In the separated case, it holds via Construction \ref{const finite model flat imm}.
\end{rem}

\subsection{Algebraic morphisms}\label{section algebraic}

Inspired the idea of representable morphisms in the context of algebraic spaces or algebraic stacks, we briefly describe an alternative way of defining "proper" morphisms of schematic spaces and, in general, extending any definition of scheme theory to the schematic setting.
\begin{defn}\label{definition algebraic schematic}
We say that a schematic morphism  $f\colon X\to Y$ is \textit{algebraic} if for any schematic morphism $Z\to Y$ such that $\Spec(Z)$ is a scheme, $\Spec(Z\times_YX)$ is a scheme. Similarly, we say that a space $X$ is algebraic if its diagonal map $X\to X\times X$ is algebraic. Given a property $\mathbf{P}$ of morphisms of schemes that is 1) stable under arbitrary base change, and 2) local in the fpqc topology; we say that $f$ is \textit{algebraically $\mathbf{P}$} if it is algebraic and, for all $Z\to Y$, it holds that $\mathbf{P}(\Spec(Z\times_YX)\to\Spec(Z))$. 
\end{defn}
The upside of this definition is that it is geometric by construction; the obvious downside is that it requires us to invoke the $\Spec$ functor, making it not intrinsic to the world of schematic spaces, and that it requires us to check a property over \textit{all $Z$ such that $\Spec(Z)$ is a scheme}, including spaces $Z$ that are not finite models of schemes. This last point is essential for the Definition to work as intended, but it makes algebraicity a very strong condition to check in practice. We continue with some facts and remarks about algebraic morphisms, without proofs. Basic properties are straightforward thanks to Proposition \ref{prop commutation}. Let $f\colon X\to Y$ be a schematic morphism:
\begin{itemize}
\item If $f$ is algebraic, then it is a scheme fibration. Although we strongly believe it to be false, it is open whether or not the converse holds; which, if it were true, would imply that all morphisms pro-locally of finite presentation are algebraic and such condition would be checked via Theorem \ref{theorem prolocal}. If it were true, its proof should be reminiscent of similar results on classical algebraic spaces, but without the finite presentation property of étale morphisms, those methods cannot be reproduced in our context.

\item If $\Spec(f)$ is a morphism of schemes---which makes $f$ automatically algebraic---, then $\mathbf{P}(\Spec(f))$ if and only if $f$ is algebraically $\mathbf{P}$. In particular, \textit{algebraic isomorphisms are exactly the qc-isomorphisms.}

\end{itemize}

For a complete exposition, we prove that, in spite of the first point above, \textit{all affine morphisms} are algebraic, which allows one to assume algebraicity without loss of generality in some contexts---for instance, studying finite covers of schematic spaces---. This relies heavily upon Theorem \ref{theorem b422} and related methods introduced in Section \ref{section cohomology} below, so the reader should skip the following proof and come back to it at the end.
\begin{thm}
An schematic morphism $f\colon X\to Y$ is affine if and only if it is algebraically affine.
\begin{proof}
The non-trivial part is the \textit{only if} implication.	 It suffices to prove that, if $S=\Spec(Z)$ is a scheme and $\A$ a quasi-coherent $\OO_Z$-algebra, then one has an isomorphism of spaces over $S$
\begin{align*}
\Spec((Z, \A))\simeq \Spec_S(\widehat{\A});
\end{align*}
where $(Z, \A)$ is the schematic space obtained by replacing the structure sheaf $\OO_Z$ with $\A$---this plays the role of the relative spectrum in the schematic world---and $\Spec_S$ denotes the ordinary relative spectrum of scheme theory. Indeed, if this holds, since $f$ is affine, it is qc-isomorphic to $(Y, f_*\OO_X)\to Y$ by Stein factorization for affine morphisms, see \cite[]{fernando schemes}; so for any $g\colon Z\to Y$ one has that $Z\times_YX$ is qc-isomorphic to $Z\times_Y(Y, f_*\OO_X)\simeq (Z, g^*f_*\OO_X)$, thus $$\Spec(Z\times_YX)\simeq \Spec((Z, g^*f_*\OO_X))\simeq \Spec_S(\widehat{g^*f_*\OO_X}),$$ which is a scheme, as desired.

To prove the initial claim, first assume that $Z$ has restrictions of finite presentation. Then $\pi_Y\colon S\to Z$ is continuous and, by Theorem \ref{theorem characterization schematic}, its image $\pi_Z(S)$ with the induced sheaf of rings is another schematic space qc-isomorphic to $Z$---in future work we will call this \textit{the concise space associated to $Y$}---, so we may assume that $\pi_Z$ is surjective. Note that $\widehat{\A}=\pi_Z^*\A$ and that both spaces in the statement are now defined by the same colimit. For a general $Z$, Corollary \ref{corollary change model} reduces the proof to the case we have just discussed. Some minor details omitted.
\end{proof}
\end{thm}

\section{Preservation of cohomology and proschemes}\label{section cohomology}

In this section we discuss a rather subtle problem that arises when one tries to extend cohomological results from scheme theory to schematic spaces. Let $X$ be a schematic space and assume that $\Spec(X)$ is a scheme. In general, $X$ is not necessarily a \textit{finite model} of a scheme in the sense of the introduction: we shall interpret it as a \textit{generalized finite model with respect to a cover by flat monomorphisms.} There is an \textit{ad hoc} way to define an equivalence of abelian categories $\Qcoh(X)\simeq \Qcoh(\Spec(X))$---restricting, in fact, to $\Coh(X)\simeq \Coh(\Spec(X))$ for \textit{any} schematic space---, but this equivalence isn't directly given by taking direct image with respect to a continuous map---or morphism of sites---in the ordinary sense, because the centre map $\pi_X$ of Theorem \ref{theorem characterization schematic} is not continuous for the Zariski topology unless $X$ has restriction maps of finite presentation. The consequence is that the $i$-th cohomology functors $H^i(\Phi)\colon \Qcoh(\Spec(X))\to \Qcoh(X)$, that are trivial for $i>0$, do not verify that $H^i(\Phi)(\M)$ is the sheafification $U\mapsto H^i(V, \M)$ for open subset $V\subseteq \Spec(X))$. In other words, exactness does not automatically imply preservation of cohomology. 

To solve this we are going to define the equivalence above in a roundabout way by means of the following results ---that we believe to hold in greater generality that the one explained here---. In what follows, let $\mathbf{AlgSchFin}$ denote the subcategory of algebraic schematic spaces and morphisms and let $\mathbf{AlgPSch}$ denote the subcategory of locally ringed spaces of \textit{algebraic proschemes}---that we will rigorously define in Definition \ref{definition proschemes}---. The proof of Theorem \ref{theorem b422} is very technical and provided at the end of the section.
\begin{thm}\label{theorem b422}
The restriction of the $\Spec$ functor $$\Spec\colon \mathbf{AlgSchFin}_{\qc}\to \mathbf{AlgPSch}$$ is an equivalence of categories.
\end{thm}
With this, it is immediate to prove the following:
\begin{cor}\label{corollary change model}
If $X$ is a schematic space such that $\Spec(X)$ is a scheme, there are schematic spaces $Z$ and $T$ such that $Z$ has restrictions of finite presentation and qc-isomorphisms $Z\leftarrow T\to X$.
\begin{proof}
Note that any such $X$ is algebraic. Let $Z$ be any finite model of the scheme $\Spec(X)$ and let $\mathrm{\Psi}\colon \Spec(Z)\to \Spec(X)$ be the natural isomorphism, which will trivially be algebraic. Theorem \ref{theorem b422} implies that $\Psi=\Spec(f/\phi)$ for some qc-isomorphisms $f\colon T\to X$ and $\phi\colon T\to Z$, so we win.
\end{proof}
\end{cor}

\begin{cor}[Cohomology-preservation Lemma]\label{corollary cohomology}
Let $X$ be a schematic space such that $\Spec(X)$ is a scheme. There is a cohomology-preserving equivalence of abelian categories $\Qcoh(X)\simeq \Qcoh(\Spec(X))$.
\begin{proof}
With the notations of Corollary \ref{corollary change model}, we obtain a functor
\begin{align*}
f_*\circ g^*\circ \pi_{Z*}\circ \Psi^*\colon \Qcoh(\Spec(X))\to \Qcoh(X),
\end{align*}
where $\pi_Z\colon \Spec(Z)\to X$ is a continuous map. It is a cohomology-preserving equivalence because all involved morphisms are so.
\end{proof}
\end{cor}

The rest of this section is mostly devoted to prove Theorem \ref{theorem b422}.
\begin{defn}
A locally ringed space $S$ is a \textit{proscheme} if there is some schematic space $X$ and an isomorphism $S\simeq \Spec(X)$. We say that $X$ is a \textit{presentation} of $S$.
\end{defn}

Proschemes are interesting objects on their own, since it can be shown that they are \textit{the most general locally affine "spaces"} for which the locally ringed space structure retains \textit{all} geometric information. They are a generalization of schemes, but in a different sense than étale-theoretic algebraic spaces, for which the above claim fails.

For any schematic space $X$ we have a natural map $\amalg_x U_x\to X$ such that the pullback $\Qcoh(X)\to \prod_x\Qcoh(U_x)\simeq \prod_x\Qcoh(\Spec(U_x))$ is faithfully exact. This map induces
\begin{align*}
\rho\colon\coprod_{x\in X}\Spec(U_x)\to\Spec(X).
\end{align*}
\begin{defn}
Given a proscheme $S\simeq\Spec(X)$, de define the category of \textit{descended quasi-coherent modules} as the fully faithful subcategory
\begin{align*}
\mathbf{dQcoh}(S):=\{\M\in\mathbf{Mod}(S):\rho^*\M\in\Qcoh(\amalg_x\Spec(U_x))\}.
\end{align*}
\end{defn}

\begin{lem}
For any schematic space $X$ with $S\simeq \Spec(X)$, the functor
\begin{align*}
\widehat{(-)}\colon \Qcoh(X)\to \mathbf{dQcoh}(S), & & \M\mapsto \lim_{x\in X}i_{x*}\widetilde{\M_x}
\end{align*}
(where $\widetilde{\M_x}$ is the quasi-coherent module on $\Spec(\OO_{X, x})$ defined by $\M_x$ and $i_x\colon\Spec(U_x)\to \Spec(X)$ is the natural flat monomorphism) defines an equivalence of categories. In particular, $\mathbf{dQoh}(S)$ is independent of the chosen	 presentation of $S$.
\begin{proof}
The quasi-inverse functor maps $\mathcal{N}\in \mathbf{dQcoh}(S)$ to $\mathcal{N}_X$ such that $\mathcal{N}_{X, x}:=\mathcal{N}(i_x(\Spec(U_x)))\simeq (i_x^*\mathcal{N})(\Spec(U_x))$.
\end{proof}
\end{lem}
\begin{cor}\label{corollary useless}
If $S\simeq \Spec(X)$ is a scheme, $$\Qcoh(X)\simeq \mathbf{dQcoh}(S)\simeq \Qcoh(S).$$
\begin{proof}
This follows from fpqc descent for schemes, \cite[Theorem 03O8]{stacks}.
\end{proof}
\end{cor}
\begin{claim}
It can be shown that, for \textit{any} $X$, the functor $\widehat{(-)}$ restricts to an equivalence $\Coh(X)\simeq \Coh(S)$. Actually, $\Coh(S)\subseteq \mathbf{dQcoh}(S)$ is a subcategory of compact generators, hence every descended quasi-coherent module is a filtered colimit of coherent modules. 
\end{claim}

As explained in the introduction of this section, the equivalence of Corollary \ref{corollary useless} fails to guarantee preservation of cohomology, so we will recover it in a different way: finding qc-isomorphisms relating $X$ to a \textit{finite model} $Z$ of $S$.
\begin{defn}
A morphism of locally ringed spaces $g\colon S\to T$ between proschemes $S$ and $T$ is said to be \textit{schematic} if $g_*$ preserves descended quasi-coherent modules.
\end{defn}

Proschemes and their schematic morphisms define a category denoted
\begin{align*}
\mathbf{PSch}.
\end{align*}

\begin{defn}\label{definition proschemes}
A schematic morphism $g\colon \Spec(X)\to \Spec(Y)$ is said to be \textit{algebraic} if for every schematic space $Z$ such that $\Spec(Z)$ is a scheme and any $Z\to Y$, the base change $\Spec(Z\times_YX)\simeq \Spec(Z)\times_{\Spec(Y)}\Spec(X)$ is a qc-qs scheme. A proscheme $S$ is \textit{algebraic} if $S\to S\times S$ is algebraic.
\end{defn}
\textit{A priori}, this definition requires fixing presentations for our spaces---but not for morphisms!---and it shall be understood as \textit{a morphism is algebraic if there is some presentation of its source and target spaces that verifies the condition}, but Theorem \ref{theorem b422} will prove that is independent of it. Since isomorphisms are always algebraic---trivially so!---, one still readily checks that algebraic proschemes and their algebraic schematic spaces define a subcategory
\begin{align*}
\mathbf{AlgPSch}\subseteq \mathbf{PSch}
\end{align*}
and, since the $\Spec$ functor commutes with finite fibered products (Proposition \ref{prop commutation}), that there is a functor
\begin{align*}
\Spec\colon \mathbf{AlgSchFin}_{\qc}\to \mathbf{AlgPSch}.
\end{align*}
Now we prove two technical Lemmas and, with them, Theorem \ref{theorem b422}.
\begin{lem}[Technical Lemma 1]\label{lemma tech 1}
Let $S$ be an \textit{affine scheme} and let $\{f_i\colon V_i\to S\}$ a finite covering by flat monomorphisms such that, for every $s\in S$, the fibered product $V^s:=\prod_{s\in V_i}V_i$---over $S$---is affine. Then, the topological finite model $X$ with respect to the images of this covering---that are not Zariski-open!---, equipped with a Zariski-non-continuous map $\pi_S\colon S\to X$ and with the sheaf of rings $\OO_{X, \pi_S(s)}:=\OO_{V^s}(V^s)$ is an affine schematic space such that $S\simeq \Spec(X)$. 
\begin{proof}
First, given any flat monomorphism of affine schemes $i\colon V\to S$, since faithfully flat monomorphisms are isomorphisms, any flat monomorphism with image $i(V)$ will share global sections with $V$. In particular, it makes sense to define $\OO_S(V):=\OO_V(V)$---this coincides with the limit of the sections over all Zariski-open neighborhoods of $i(V)$---and, thus, we shall identify these flat monomorphisms with their images. Now, the key to this proof is the additional hypothesis on the covering: given a covering $\{W_k\to V\}$ by flat monomorphisms and such every intersection $W_k\times_S W_{k'}$ is covered by a finite number of $W_{k''}$, the sheaf condition translates to $\OO_S(V)\simeq \lim\OO_S(W_{k''})$. In our context,  $\{V^s\to S\}_{s\in S}$---for a finite number of $s\in S$---plays this role.

Define the "finite model" $X$ as in the statement. By definition of $\Spec(X)$, we have a map $f\colon \Spec(X)\to S$ of locally ringed spaces, surjective and such that $f^{-1}(V^s)=\pi_X^{-1}(U_{\pi_S(s)})=\Spec(U_{\pi_S(s)})$. Note that $X$ is clearly pseudoschematic, because base changes of flat ring epimorphisms are stable under base changes, and that the composition $\pi_S\circ f\colon \Spec(X)\to X$ is surjective. By Theorem \ref{theorem characterization schematic} we only need to prove that $f$ is an isomorphism. 

First, $f$ is an homeomorphism by faithfully flat descent for schemes, see \cite{stacks}[Lemma 02JY], but by the reasons discussed above, we also know that $\OO_S(S)\simeq \lim_s\OO_S(V^s)\simeq\OO_{\Spec(X)}(\Spec(X))$, and by the universal property of the prime spectrum
\begin{align*}
\Hom_{\mathbf{LRS}}(\Spec(X), \Spec(\OO_S(S)))\simeq \Hom_{\mathbf{CRing}}(\OO_S(S), \OO_{\Spec(X)}(\Spec(X))
\end{align*}
we conclude that $f$ is isomorphic to the natural affinization morphism in $\mathbf{LRS}$, namely $\Spec(X)\to \Spec(\OO_{\Spec(X)}(\Spec(X))$, which is isomorphic to an homeomorphism, and thus necessarily an isomorphism. 
\end{proof}
\end{lem}
\begin{lem}[Technical Lemma 2]\label{lemma tech 2}
Given a qc-qs scheme $S$ and a finite covering by flat monomorphisms $\{W_i\to S\}$ with $W_i$ qc-qs schemes, there is a finite covering by flat monomorphisms verifying the local affinity condition of Lemma \ref{lemma tech 1} and refining $\{W_i\to S\}$---in the sense of the intersections of neighborhoods of each point of $S$---.
\begin{proof}
It is enough to consider a finite and locally affine open cover $\{U_j^i\}_j$ of each $W_i$---which exists because the $W_i$ are qc-qs---and take $\{U_j^i\to S\}_{i, j}$. If $S$ is semiseparated---affine diagonal---, then this cover is automatically locally affine. In general, the intersections $U_j^i\times_SU_{j'}^{i'}\to S$ are again quasi-compact, so each one of them can be covered by a finite number of affine opens $V^{iji'j'}_k$. Since intersections of affine open subschemes within affine schemes are affine, a straightfoward check proves that the cover $\{V^{iji'j'}_k, U_j^i\to S\}$ verifies the desired condition.
\end{proof}
\end{lem}
\begin{proof}[Proof of Theorem \ref{theorem b422}]
The functor $\Spec$ is faithful by maximality of the localization by qc-isomorphisms and essentially surjective by definition of algebraicity of proschemes. There only remains to check that it is full. Let $g\colon \Spec(X)\to \Spec(Y)$ be an algebraic schematic morphism and assume that $\Spec(X)$ is an affine scheme. Let $\{\Spec(U_y)\}_{y\in Y}$ the natural cover of $\Spec(Y)$ by flat monomorphisms and let $\{g^{-1}(\Spec(U_y)\to \Spec(X)\}$ denote its pullback, which is a cover by flat monomorphisms of locally ringed spaces such that $g^{-1}(\Spec(U_y))$ are qc-qs schemes for all $y$ by the algebraicity hypothesis.

By Lemma \ref{lemma tech 2} we find a locally affine open cover of $\Spec(X)$ refining $\{g^{-1}(\Spec(U_y))\}_y$ and, thus, obtain a finite model $\Spec(X)\to Z$, which is schematic by Lemma \ref{lemma tech 1}, and schematic morphisms
\begin{align*}
(\star, \OO_X(X))\leftarrow Z\to Y.
\end{align*}
Taking $Z':=Z\times_{(\star, \OO_X(X))}X$ we obtain $f/\phi\colon X\leftarrow Z'\to Y$ verifying $\Spec(f/\phi)\simeq g$ as desired. \textit{Note that we have not used schematicity of $g$; this proves that such an algebraic morphism is automatically schematic}.

Now, for a general $\Spec(X)$, since $X\simeq \colim_x U_x$ as ringed spaces in $\mathbf{AlgSchFin}_{\qc}$ and $\Spec(X):=\colim_x\Spec(U_x)$ as locally ringed spaces by definition and in $\mathbf{AlgPSch}$, by the universal property of colimits and the affine case discussed above, we have a chain of bijections
\begin{align*}
&\Hom(X, Y)\simeq \Hom(\colim_x U_x, Y)\simeq \lim_x\Hom(\Spec(U_x), \Spec(Y))\simeq \\&\simeq \Hom(\colim_x\Spec(U_x), \Spec(Y))\simeq \Hom(\Spec(X), \Spec(Y)),
\end{align*}
which concludes the proof.
\end{proof}

\section{Finiteness of cohomology}

We assume that all spaces are Noetherian. We denote by $D_c(X)$ the derived category of complexes of modules with coherent cohomology on $X$. We want to prove finiteness of cohomology schematic spaces. A proof via finite versions of Chow's Lemma and Dévissage can be constructed, but our intention is to show it by means of  the classic result for scheme theory. The argument is a simple one of local nature enabled by our work so far:

\begin{thm}\label{theorem main}
If $f\colon X\to Y$ is a pro-locally proper \textit{or} algebraically proper schematic morphism, its derived direct image preserves coherence, that is, $\mathbb{R}f_*\colon \mathbf{Coh}(X)\to D_c(Y)$.
\begin{proof}
It suffices to see that, for all $i\geq 0$ and any coherent module $\M$ on $X$, $R^if_*\M$ is coherent, i.e. that its stalks are finitely generated---because it is quasi-coherent by schematicity as defined in \cite{fernando schemes}---. If $f$ verifies either of the hypothesis, since qc-isomorphisms preserve both categories of quasi-coherent modules and their cohomology, we can assume that $f$ is a scheme fibration and such that $f^y\colon\Spec(f^{-1}(U_y))\to\Spec(U_y)$ are proper morphisms of schemes for all $y\in Y$. Now, by Corollary \ref{corollary cohomology}, we have
\begin{align*}
&(R^if_*\M)_y\simeq H^i(f^{-1}(U_y), \M_{|U_y})\simeq H^i(\Spec(f^{-1}(U_y)), \widehat{\M_{|U_y}})\simeq \\ & \simeq (R^if^y_*\widehat{\M_{|U_y}})(\Spec(U_y)),
\end{align*}
---where $\widehat{(-)}$ here denotes the functor of Corollary \ref{corollary cohomology}, which can be seen to coincide with the one we defined \textit{ad hoc} and is compatible with restrictions, although such facts are not needed here---which is finitely-generated by Grothendieck's finiteness theorem for proper morphisms of schemes.
\end{proof}
\end{thm}
 
In a similar manner, one can also use this technique to prove schematic generalizations for the standard consequences of this finiteness result, for example a Formal Functions Theorem or a Stein's Factorization Theorem for pro-locally proper morphisms. For the former, we recommend following the construction of completions given in \cite{fernando completaciones}, which greatly simplifies things in the schematic setting; for the latter, one has to introduce notions of geometric connectedness fibers, but that is relatively straightfoward using the theory introduced in \cite{paper grupo etale}
. We omit them for the time being because they contain essentially no novel ideas. Stein Factorization can then be used to prove a relative homotopy exact sequence for the étale fundamental group of schematic spaces, as introduced in \cite{paper grupo etale}.
\medskip

Note that our result slightly generalizes the classic one for schemes to a topology \textit{a priori} weaker than the Zariski topology: to pro-locally proper morphisms of schemes as in Definition \ref{prolocal schemes}
\begin{cor}
If $g\colon S\to T$ is a pro-locally proper morphism of Noetherian schemes, then $R^ig_*$ preserves coherence for all $i\geq 0$.
\end{cor}

\end{document}